\documentclass[12pt]{amsart}
\usepackage[usenames,dvipsnames,svgnames,table]{xcolor}

\usepackage{amssymb,amsmath,url,graphicx}
\usepackage[all]{xy}
\usepackage{xcolor}
\usepackage{enumitem}

\newif\ifdebug                                                                 %
\debugfalse                                                   

\newcommand    {\ynote}[1]   {\ifdebug {{\color{olive}{{#1}}}} \fi}

\newcommand{\printname}[1]
    {\smash{\makebox[0pt]{\hspace{-1.0in}\raisebox{-4pt}{\tiny #1}}}}
\newcommand{\labell}[1] {\ifdebug {\label{#1}\printname{#1}}
                        \else    {\label{#1}} \fi}

\swapnumbers
\numberwithin{equation}{section}
\newtheorem {Theorem}[equation]         {Theorem}

\newtheorem {Proposition}  [equation]   {Proposition}
\newtheorem {thm}[equation]         {Theorem}
\newtheorem*{thm-nn}{Theorem}  
\newtheorem {lemma}[equation]           {Lemma}
\newtheorem {corollary} [equation]      {Corollary}
\newtheorem {prop}  [equation]   {Proposition}

\theoremstyle{definition}

\newtheorem{defn}[equation]{Definition}

\theoremstyle{remark}
\newtheorem{Remark}[equation]{Remark}
\newtheorem{rmk}[equation]{Remark}
\newtheorem*{Remark*}{Remark}
\newtheorem*{rmk*}{Remark}
\newtheorem{Example}[equation]{Example}

\def\eor{\unskip\ \hglue0mm\hfill$\diamond$\smallskip\goodbreak}

\DeclareMathOperator{\RR}{RR} 
\def \BS {\text{BS}}
 
\DeclareMathOperator{\Todd}{Todd}
\DeclareMathOperator{\vol}{vol}
\DeclareMathOperator{\hol}{hol}
\DeclareMathOperator{\av}{av}

\DeclareMathOperator{\Aut}{Aut}
\DeclareMathOperator{\PD}{PD}
\DeclareMathOperator{\Hom}{Hom}
\DeclareMathOperator{\GL}{GL}

\def \aff {{\operatorname{aff}}}

\def \dR {{\operatorname{dR}}}
\DeclareMathOperator{\Inv}{inv}

\newcommand{\incl}{\iota}  

\newcommand{\C}{\mathbb{C}}
\newcommand{\R}{\mathbb{R}}

\newcommand{\Z}{\mathbb{Z}}
\newcommand{\T}{\mathbb{T}}

\newcommand{\st}{\mid}  
\def \Lt {{\mathfrak t}}

\newcommand{\del}{\partial}

\newcommand{\cross}{\times}
\newcommand{\abs}[1]{\lvert#1\rvert}

\newcommand{\tensor}{\otimes}

\newcommand{\dd}[2]{\frac{\del #1}{\del #2}}   

\newcommand{\inv}{^{-1}} 

\newcommand{\into}{\hookrightarrow}

\newcommand{\define}[1]{\textsl{#1}}
\newcommand{\iia}{integral-integral affine}  

\def \ol {\overline}

\DeclareFontFamily{U}{MnSymbolC}{}
\DeclareSymbolFont{MnSyC}{U}{MnSymbolC}{m}{n}
\DeclareFontShape{U}{MnSymbolC}{m}{n}{
    <-6>  MnSymbolC5
   <6-7>  MnSymbolC6
   <7-8>  MnSymbolC7
   <8-9>  MnSymbolC8
   <9-10> MnSymbolC9
  <10-12> MnSymbolC10
  <12->   MnSymbolC12}{}
\DeclareMathSymbol{\contract}{\mathbin}{MnSyC}{'270}

\begin{document}

\title[Integral-integral affine geometry and Riemann--Roch] 
{Integral-integral affine geometry, geometric quantization, and
Riemann--Roch}

\author{Mark D. Hamilton}
\address{Department of Mathematics and Computer Science,
  Mount Allison University, Sackville, NB, Canada}
\email{mhamilton@mta.ca}

\author{Yael Karshon}
\address{School of Mathematical Sciences, Tel-Aviv University, Tel-Aviv, Israel,
and Department of Mathematics, University of Toronto, Toronto, ON, Canada}
\email{yaelkarshon@tauex.tau.ac.il, karshon@math.toronto.edu}

\author{Takahiko Yoshida}
\address{Department of Mathematics, School of Science and Technology, 
Meiji University, Tokyo, Japan}
\email{takahiko@meiji.ac.jp}

\date{\today}




\begin{abstract}
We give a simple proof that, for a pre-quantized compact symplectic
manifold with a Lagrangian torus fibration, its Riemann--Roch number
coincides with its number of Bohr--Sommerfeld fibres.
This can be viewed as an instance of the ``independence
of polarization" phenomenon of geometric quantization.
The base space for such a fibration acquires a so-called
integral-integral affine structure.
The proof uses the following simple fact, whose proof is
trickier than we expected:  on a compact integral-integral affine manifold,
the total volume is equal to the number of integer points.
\end{abstract}

\maketitle

\setcounter{tocdepth}{1}
\tableofcontents

\ynote{2020 math subj class}

\section{Introduction}
\labell{sec:intro}

The context of this paper lies within geometric quantization.
``Quantization'' is the art of getting from a mathematical model 
for a classical mechanical system 
to a mathematical model for a corresponding quantum mechanical system;
geometric quantization does so using the geometry of the classical system.  
Recipes for geometric quantization rely on additional auxiliary structure,
called a \emph{polarization,}
which is heuristically a ``choice of half the variables,''
analogous to how classical phase space has both position and momentum coordinates,
but quantum ``wave functions'' are functions of position only. 
A central theme in geometric quantization, 
of which there is ample evidence and only partial understanding,
is ``independence of polarization,''
by which certain features of the quantized space
are independent of the choice of polarization.

The following result
can be viewed as an instance of this ``independence of polarization''
phenomenon
(see Section~\ref{sec:defns} for details and definitions):
\begin{quotation}
Let $(M,\omega)$ be a compact symplectic manifold,
let $\pi \colon M \to B$ be a (regular) Lagrangian fibration,
and let $L \to M$ be a prequantization line bundle.
Then the Riemann--Roch number of $(M,\omega)$ is equal
to the number of Bohr--Sommerfeld fibres of $L \to M \to B$.
\end{quotation}
This result is Theorem \ref{thm:RR=BS} below, 
to which we refer as the ``Upstairs Theorem''.
In principle, this result has been known to experts for many years;
it appeared as Corollary~6.12 
in the paper~\cite{fujita-furuta-yoshida} by Fujita, Furuta, and Yoshida, 
who in turn attribute it to Andersen's paper~\cite{andersen}.
In both cases, the proof passes through a more complicated theorem, and
our original purpose in writing this paper
was to give a simple, direct proof of this result.
(For more details on the relation of our work with the existing 
literature, see Section~\ref{sec:literature}.)

Our proof of the Upstairs Theorem 
relies on the following result from ``integral-integral affine geometry''
(see Section~\ref{sec:downstairs} for details and definitions),  
which on the surface has nothing to do with geometric quantization:
\begin{quotation}
Let $B$ be a compact integral-integral affine manifold.
Then the volume of $B$ is equal to the number of integral points in $B$.
\end{quotation}
This result is Theorem~\ref{thm:volB=BZ} below, 
to which we refer as the ``Downstairs Theorem''.

We were not able to find this result in the literature,
and in trying to write down a simple proof, we were surprised to find 
that it was trickier than expected.
For a while, we toyed with the idea of flipping the logic:
deducing the Downstairs Theorem from the Upstairs Theorem
and changing the title of our paper to 
``An inappropriate proof, using geometric quantization,
of a simple fact from integral-integral affine geometry.''
Eventually, it turned out that we do not need geometric quantization
to prove the Downstairs Theorem.
But we do use Lagrangian torus fibrations, and a 
not-entirely-obvious cohomological argument 
(see Theorems~\ref{thm:invariant-representative} and \ref{p:poincare}).
These arguments may be known to experts
(see Section~\ref{sec:literature}),
but we were not able to find an explicit proof in the literature.

One thing that makes the Downstairs Theorem interesting 
is that, although the base $B$ is compact,
we do not know if its affine connection is geodesically complete.
This is a special case of the Markus conjecture;
see \S\ref{sec:downstairs}.

An alternative approach to the Downstairs Theorem, 
suggested to us by Yiannis Loizides, 
is to decompose $B$ into simple integral convex polytopes $B_i$
and to apply to each $B_i$ the exact Euler-MacLaurin formula of \cite{KSW}.
Such an approach might not be simpler than ours,
but it may lead to a more general result, 
in which we allow $\pi \colon M \to B$ to be a completely integrable
system with elliptic (``locally toric'') singularities
over a manifold-with-corners $B$.
Such systems are studied in~\cite{mol} and~\cite{fernandes-mol}.

The organization of this paper is as follows. 
In Section~\ref{sec:defns}, we state our Upstairs Theorem
as Theorem~\ref{thm:RR=BS} and give some necessary definitions.  
In Section~\ref{sec:downstairs}, 
we recall facts about integral affine structures 
and define \emph{integral-integral affine structures,}
leading to the statement of our Downstairs Theorem 
in Theorem~\ref{thm:volB=BZ}.
In Section~\ref{sec:defns 2},
we recall how a Lagrangian torus fibration 
induces an integral affine structure on its base,
and we show --- using ``enhanced Arnol'd--Liouville charts'' --- 
how a prequantized Lagrangian torus fibration 
induces an integral-integral affine structure on its base.
In Section~\ref{sec:RR=BS}, we give our short proof of 
the Upstairs Theorem modulo the Downstairs Theorem.
In Section~\ref{sec:lie-torus-bundles},
we describe the integral lattice in $TB$ and period lattice in $T^*B$
when $B$ is integral affine,
and the affine lattice in $TB$ when $B$ is integral-integral affine. 
In Section~\ref{sec:local torus actions}, 
we introduce \emph{local torus actions,}
and we show that, with a local torus action,
every de Rham cohomology class is represented 
by an invariant form.
Finally, in Section~\ref{sec:pf-volB=BZ} we use local torus actions
and the aforementioned cohomological argument 
to prove the Downstairs Theorem.
Section~\ref{sec:dual-torus-fibration}
contains a discussion of dual torus fibrations that is not necessary 
for our proof but that provides geometric context.
And in Section~\ref{sec:literature} we comment on the connection
of our work to some existing literature.

\smallskip

Here is a list of the two main theorems,
and the lemmas that allow us to deduce the Upstairs Theorem
from the Downstairs Theorem.

\begin{center}
\renewcommand{\arraystretch}{1.5}
\begin{tabular}{|lc|}
\hline
Theorem~\ref{thm:RR=BS} (Upstairs): & $\RR(M)=|\BS|$ \\ \hline
Theorem~\ref{thm:volB=BZ} (Downstairs): & $\vol(B)=|B_\Z|$ \\ \hline \hline
Lemma~\ref{lemma:volM=volB}: & $\vol(M)=\vol(B)$ \\ \hline
Lemma~\ref{lemma:prequant-implies-IIA}: & $\BS = B_\Z$ \\ \hline
Lemma \ref{lemma:RR=volM}: & $\RR(M)=\vol(M)$ \\ \hline
\end{tabular}
\end{center}

\medskip


\subsection*{Acknowledgements}
We are grateful for discussions with Bill Goldman,  
Yiannis Loizides, 
Oliver Goertches, 
Alex Lubotzky, 
Yuichi Nohara, 
and Louis Ioos. 

Yael Karshon and Mark Hamilton acknowledge the support
of the Natural Sciences and Engineering Research Council of Canada (NSERC).
Yael Karshon's research is also partly funded by the 
United-States -- Israel Binational Science Foundation.
Takahiko Yoshida's work is supported by Grant-in-Aid for Scientific
Research (C) 15K04857 and 19K03479.

\section{Upstairs --- geometric quantization}
\labell{sec:defns}

In this section we give the definitions necessary to state our 
``Upstairs Theorem,'' as well as provide some context, 
and state the theorem as Theorem~\ref{thm:RR=BS}.
We prove this theorem in Section~\ref{sec:RR=BS}, assuming 
our ``Downstairs Theorem,'' which we state in Section~\ref{sec:downstairs}
and prove in Section~\ref{sec:pf-volB=BZ}.

For our purposes, geometric quantization is a recipe that associates
to a compact symplectic manifold $(M,\omega)$
a finite dimensional vector space $\mathcal{Q}(M)$
(or more generally a virtual vector space, i.e.,
a formal difference of two vector spaces),
which is constructed from sections of a particular complex line bundle
$L\to M$.  
One part of the recipe of geometric quantization is a ``polarization,''
of which the most commonly considered types are as follows.
A \define{K\"ahler polarization} is 
given by a compatible complex structure on $M$,
which makes $L$ into a holomorphic line bundle; 
then $\mathcal{Q}(M)$ is the space of holomorphic sections of $L$.\footnote{
  Heuristically, we can think of a holomorphic section as a section that
  ``depends on  $z$  and not on  $\ol{z}$.''}
A \define{real polarization} is given by a foliation of $M$
into Lagrangian submanifolds,
usually assumed to be the fibres of a fibration  $\pi \colon M \to B$.
In this case, $\mathcal{Q}(M)$ is constructed from sections
of $L \to M$ that are covariant constant along the fibres of $\pi$, 
and can typically be described in terms of
\emph{Bohr--Sommerfeld fibres,} defined further below.\footnote{
  Heuristically, we can think of a section that is covariant constant
  along the fibres of $\pi$ as a section that
  ``depends on the base variables and not on the fibre variables.''}
One natural question that arises is 
``independence of polarization'': 
Given two polarizations on the same $M$, are the resulting quantizations
the same?

\subsection*{Riemann--Roch numbers}\  

Recall that the Todd class associates to each rank $n$ complex vector bundle
$E \to M$ a cohomology class of mixed degree on $M$,
associated through the Chern--Weyl recipe to the Taylor series at $0$
of the function $ \prod_{i=1}^n \frac{x_i}{1-e^{-x_i}}$.
Given a compact symplectic manifold $(M,\omega)$, we define its
\define{Riemann--Roch number} by
\begin{equation}\labell{eq:RR-defn}
  \RR(M,\omega) = \int_M \exp(\omega) \Todd(TM,J),
\end{equation}
where $J$ is any compatible almost complex structure.
(Different choices of compatible almost complex structures~$J$
give isomorphic complex vector bundles $(TM,J)$,
hence the same Todd class.)

We can interpret the Riemann--Roch number as the dimension of a
quantization of $(M,\omega)$.  
When $\omega$ is the curvature of a complex Hermitian line bundle $L \to M$,
we can declare the quantization space to be the virtual vector space
obtained as the formal difference of the kernel and cokernel 
of the corresponding Dirac--Dolbeault operator.
See, e.g., Duistermaat's book \cite[Chapter 15]{duistermaat}.
The Riemann--Roch number then computes the index of this operator,
i.e., the difference of the dimensions of the kernel and cokernel;
in particular, it's an integer.
In the presence of a compatible complex structure on $M$,
the line bundle $L \to M$ becomes holomorphic,
and this virtual vector space 
coincides with the alternating sum of the cohomologies
of the sheaf of holomorphic sections of $L$.
When the higher cohomology vanishes
(such as when we pass to $L^{\otimes k}$ for sufficiently large $k$),
we are left with the space of holomorphic sections,
and in this case the Riemann--Roch number gives the dimension of the space of
holomorphic sections, which can be viewed as the dimension of
the quantization $\mathcal{Q}(M)$.

\subsection*{Bohr--Sommerfeld sets}\ 

Let $(M,\omega)$ be a symplectic manifold,
and let $L \to M$ be a prequantization line bundle
with connection,
i.e., $L$ is a complex Hermitian line bundle,
equipped with a connection whose curvature is $\omega$.\footnote{
  There are various different conventions in the literature for the curvature
  of a prequantization line bundle: often it includes factors of
  $i$, $2\pi$, and/or $\hbar$.
  See Remark~\ref{rmk:conventions}.}
For any Lagrangian submanifold $N$ of $M$, the pullback of $L$ to $N$
is a line bundle with a flat connection.
The Lagrangian $N$ is \define{Bohr--Sommerfeld}
if this connection is trivializable.
Equivalently, $N$ is Bohr--Sommerfeld 
if the holonomy of the prequantization
connection is trivial around every loop in $N$.
Given a Lagrangian fibration $\pi \colon M \to B$,
the corresponding \define{Bohr--Sommerfeld points} 
are those points of $B$ whose preimage in $M$ is a Bohr--Sommerfeld Lagrangian;
the set of such points is the \define{Bohr--Sommerfeld set},
which we denote $\BS$.

In this case the quantization of $M$
with respect to the real polarization defined by the Lagrangian fibration $\pi$
is constructed from sections
of $L\to M$ that are covariant constant along the fibres of $\pi$.
There are several recipes for doing so; as one example,
\'Sniatycki~\cite{sniatycki} constructs this quantization 
as a cohomology group of the {sheaf} of such sections.
(See~\S3 of~\cite{hamilton-monod} for a fuller discussion of quantization
using a real polarization.)
Typically, we can interpret the Bohr--Sommerfeld set as 
indexing a set of basis elements of a quantization of $(M,\omega)$, so that 
the number of Bohr--Sommerfeld points in~$B$ 
gives the dimension of the quantization of $M$. 
For the purposes of this paper, 
we \emph{define} the dimension of the Bohr--Sommerfeld quantization in this way.

The main theorem of this paper is the following:
\begin{thm}[``Upstairs theorem'']
\labell{thm:RR=BS}
Let $(M,\omega)$ be a compact symplectic manifold,
let $\pi \colon M \to B$ be a (regular) Lagrangian fibration,
and let $L \to M$ be a prequantization line bundle.
Then the Riemann--Roch number of $M$ is equal to the number of
Bohr--Sommerfeld fibres of $L \to M \to B$:
\[   \RR(M,\omega) = \abs{\BS\,}. \]         
\end{thm}
Theorem~\ref{thm:RR=BS} can be viewed as an instance 
of ``independence of polarization,''
in the sense that 
the vector space that is associated to $L \to M$
through Dirac--Dolbeault quantization has the same dimension
as the vector space that is associated to $L \to M$
through Bohr--Sommerfeld quantization.
In Section~\ref{sec:RR=BS} we will reduce Theorem~\ref{thm:RR=BS} 
to Theorem~\ref{thm:volB=BZ}; 
we will prove Theorem~\ref{thm:volB=BZ} in Section~\ref{sec:pf-volB=BZ}.

One corollary of the arguments leading to Theorem~\ref{thm:RR=BS} is that 
different Lagrangian torus fibrations on a symplectic manifold $M$
have the same number of Bohr--Sommerfeld points:

\begin{corollary} \labell{corollary to upstairs}
  Let $(M,\omega)$ be a compact symplectic manifold,
  let $L \to M$ be a prequantization line bundle, 
  and let $\pi_1 \colon M \to B_1$ and  $\pi_2 \colon M \to B_2$
  be Lagrangian fibrations,
  with corresponding Bohr--Sommerfeld sets $\BS_1$ and $\BS_2$.
  Then $\abs{\BS_1} = \abs{\BS_2}$.
\end{corollary}

\begin{proof}
By putting together Lemma~\ref{lemma:volM=volB}, 
Lemma~\ref{lemma:prequant-implies-IIA}, and Theorem~\ref{thm:volB=BZ}, 
we see that the number of Bohr--Sommerfeld points in 
either $B_1$ or $B_2$ equals the volume of $M$.  
\end{proof}

In the context of quantization, this means that two real polarizations
on the same manifold $(M,\omega)$ will yield Bohr--Sommerfeld
quantizations of the same dimension.
For example, see the discussion of different real polarizations
on the Kodaira--Thurston manifold in~\cite{mark-zoe}.

We are grateful to Jonathan Weitsman for the following observation,
which follows by the same argument as Corollary~\ref{corollary to upstairs}:

\begin{corollary}\labell{cor:ind-of-line-bundle}
  Let $(M,\omega)$ be a compact symplectic manifold, and
  $\pi\colon M \to B$ a (regular) Lagrangian fibration. 
  Let $L_1 \to M$ and $L_2 \to M$ be prequantization line bundles with
  corresponding Bohr--Sommerfeld sets $BS_1$ and $BS_2$.
  Then $\abs{\BS_1} = \abs{\BS_2}$.
\end{corollary}

If a prequantizable symplectic manifold $M$ is simply-connected, then 
it has a unique prequantization line bundle up to equivalence.
(Indeed, if $L$ and $L'$ are prequantization line bundles,
then $L^{-1} \otimes L'$ is a line bundle with a flat connection,
which is trivial if $H_1(M)=0$.
See \cite[Theorem 2.2.1]{Kost}.)
%
Corollary~\ref{cor:ind-of-line-bundle} shows that in the case of Lagrangian
torus fibrations (which typically have $H_1\neq 0$),
the (dimension of the) real quantization is independent of the choice of $L$.
A similar result will hold for any quantization whose dimension is given
by the Riemann--Roch number.

We are similarly grateful to Daniele Sepe for the following observation,
which says that the quantization is unchanged by adding a ``magnetic term'':

\begin{corollary}
  Let $(M,\omega)$ be a compact symplectic manifold,
  $\pi\colon M \to B$ a (regular) Lagrangian fibration, and
  $L\to M$ a prequantization line bundle.
  Suppose $\beta$ is a closed integral 2-form on $B$, and 
  suppose $L'\to M$ is a prequantization line bundle with respect to
  $\omega + \pi^*\beta$.
  Then the Bohr--Sommerfeld sets of $L$ and $L'$ 
have the same number of points.
\end{corollary}

\begin{proof}
If $L' = L \otimes \pi^* L_B$
where $L_B \to B$ is a line bundle with connection whose curvature
is $\beta$ (which exists by \cite[Theorem 2.1.1]{Kost}), 
then $L'$ is a prequantization of $(M,\omega+\pi^*\beta)$,
and its Bohr--Sommerfeld set is equal to that of $L$.
Together with Corollary~\ref{cor:ind-of-line-bundle},
we obtain the result for general $L'$.
\end{proof}

\section{Downstairs --- integral-integral affine geometry}
\labell{sec:downstairs} 

In this section, we review some facts about integral affine structures, 
and we define the notion of ``integral-integral affine structure'',
leading to the statement of our Downstairs Theorem, 
Theorem~\ref{thm:volB=BZ}.

\subsection*{Affine structures} \

An \define{affine map} on $\R^n$ is a map of the form 
$x\mapsto Ax+b$ for some $A\in \GL_n(\R)$ and $b\in \R^n$.
An \define{affine atlas} on a manifold $B$ is a collection
of charts $\phi_\alpha \colon U_\alpha \to \Omega_\alpha \subset \R^n$
whose domains $U_\alpha$ cover $B$ and whose transition maps 
$\phi_\beta \circ\phi_\alpha\inv$ are locally affine.
An \define{affine structure} is a maximal affine atlas.
Such a structure induces a torsion-free flat connection on $TB$,
by declaring the coordinate vector fields to be horizontal.
    
\subsection*{Integral affine structures} \

An \define{integral affine map} on $\R^n$ is a map of the form 
$x\mapsto Ax+b$ for some $A\in \GL_n(\Z)$ and $b\in \R^n$.
Equivalently, it is an affine map whose linear part takes the lattice $\Z^n$
onto itself.
An \define{integral affine atlas} on a manifold $B$ is a collection
of charts $\phi_\alpha \colon U_\alpha \to \Omega_\alpha \subset \R^n$ 
whose domains cover $B$ and whose transition maps 
$\phi_\beta \circ\phi_\alpha\inv$ are locally integral affine.
An \define{integral affine structure} is a maximal integral affine atlas.
Such a structure induces a smooth measure on $B$,
coming from the usual Lebesgue measure on $\R^n$.
We can then refer to the total volume, $\vol(B)$.

\subsection*{Integral-integral affine structures} \

An \define{integral-integral affine map} of $\R^n$ 
is a map of the form $x \mapsto Ax+b$ 
for some $A\in \GL_n(\Z)$ and $b\in \Z^n$.
Equivalently, it is an affine map that takes the lattice $\Z^n$ onto itself.
An \define{integral-integral affine atlas} on a manifold $B$
is a collection of charts 
$\phi_\alpha \colon U_\alpha \to \Omega_\alpha \subset \R^n$
whose domains cover $B$ and whose transition maps
$\phi_\beta \circ\phi_\alpha\inv$ are locally integral-integral affine.
An \define{integral-integral affine structure} is a maximal
integral-integral affine atlas.
Such a structure determines a subset $B_\Z$ of $B$,
defined as the set of those points $b \in B$
that are sent to $\Z^n$ by some, hence all, integral-integral charts
that contain $b$.
We call it the set of \define{integral points} in $B$.

\subsection*{Examples}\  

A prototypical example of an integral-integral affine manifold is
the quotient of $\R^n$ by a group $\Gamma$ of integral-integral affine
transformations that acts freely and properly.
Specific examples include the torus, where $\Gamma$ is generated
by the maps 
$$ x_1 \mapsto x_1+1 , \quad \ldots , \quad x_n \mapsto x_n+1 ;$$
the Klein bottle, where $n=2$ and $\Gamma$ is generated by the maps
$$ x_1 \mapsto x_1+1 , \quad (x_1,x_2) \mapsto (-x_1,x_2+1);$$
and Kodaira-Thurston-like manifolds, with $\Gamma$ generated by the maps
$$ (x_1,x_2,x_3) \mapsto (x_1+x_2,x_2,x_3+1),
\quad \text{ and } \quad x_j \mapsto x_j +1 \text{ for all $j \neq 3$} .$$
We ask:
\begin{quotation}
Does every compact integral-integral affine manifold
arise as a quotient of $\R^n$
by a free and proper action of a discrete group 
of integral-integral affine maps?
\end{quotation}
This appears to be an open question.
It is a special case of the Markus conjecture,
which posits that if a closed affine manifold possesses
an atlas whose transition maps 
all have $\det A = \pm 1$
then it is geodesically complete.
For a survey of the Markus conjecture, see \cite[Chap.~11]{goldman}. 

\subsection*{Volume and lattice points} \

Recall that an integral affine structure on a manifold $B$
determines a smooth measure on $B$,
and an \emph{integral-integral affine} structure
also determines a notion of ``integer points'' in~$B$.
We are now ready to state 
the result in ``integral-integral affine geometry''
that we will need in our proof of Theorem~\ref{thm:RR=BS}:

\begin{thm}[``Downstairs theorem'']
\labell{thm:volB=BZ}
Let $B$ be a compact integral-integral affine manifold.
Then the volume of $B$ is equal to the number of integral points in $B$:
  \[ \vol(B) = \abs{B_\Z}.\]
\end{thm}
We will prove this theorem in Section~\ref{sec:pf-volB=BZ}.

\section{Arnol'd--Liouville}
\labell{sec:defns 2}

We begin this section by recalling some properties of proper
Lagrangian fibrations.
First, we recall the Arnol'd--Liouville theorem.
Next, we recall that a proper Lagrangian fibration
induces on its base an integral affine structure,
whose induced measure is the push-forward of Liouville measure.
We then prove an ``enhanced Arnol'd--Liouville Theorem,''
giving a local model for a pre-quantized proper Lagrangian fibration. 
Finally, we show that a pre-quantized proper Lagrangian fibration
induces on its base an integral-\emph{integral} affine structure,
whose lattice points coincide with the Bohr--Sommerfeld points.
This provides preparation for Section~\ref{sec:RR=BS},
where we prove our Upstairs Theorem (``$\RR = |\BS|$''),
assuming our Downstairs Theorem (``$\vol(B)=|B_\Z|$'').
We defer the proof of our Downstairs Theorem 
to Section~\ref{sec:pf-volB=BZ}.

\subsection*{Lie tori and smooth tori} \

We will encounter some tori.
For a fixed $n$, we will use $\T^n$ to denote the standard torus $(S^1)^n$.
To avoid ambiguity, we will use the terms \define{Lie torus}
for a Lie group that is isomorphic (as a Lie group) to $\T^n$
for some $n$,
and \define{smooth torus} for a manifold that is diffeomorphic to $\T^n$
for some $n$.

\subsection*{Arnol'd--Liouville}\ 

A map from a symplectic manifold $(M,\omega)$ to a manifold $B$
is a \define{Lagrangian fibration} if
it is a (locally trivializable) fibre bundle 
whose fibres are Lagrangian submanifolds of $M$. 
It is a \define{Lagrangian torus fibration}
if, additionally, its fibres are (smooth) tori.\footnote{
People often study ``singular Lagrangian fibrations'',
especially in the study of integrable systems.  
In this paper we only consider ``regular'' fibrations,
i.e., actual fibrations.
}
The Arnol'd--Liouville theorem implies that every proper Lagrangian fibration
with connected fibres is a Lagrangian torus fibration.
Moreover, this theorem gives local ``action-angle coordinates''
on such fibrations.  
We express such coordinates in terms of what we call 
``Arnol'd--Liouville charts.''

\begin{defn}
Given a proper Lagrangian fibration $\pi\colon (M,\omega) \to B$ 
with connected fibres,
an \define{Arnol'd--Liouville chart} consists of:
\begin{itemize}
\item a connected open subset $U$ of $B$
\item a diffeomorphism of $U$ with an open subset $\Omega$ of $\R^n$
\item a lifting of this diffeomorphism to a symplectomorphism
from $\pi\inv(U)$ to $\Omega\cross \T^n$ with the
standard symplectic form $\sum dx_j \wedge dt_j$,
with $t_j$ mod $1$ coordinates on the $\T^n$ factor.
\end{itemize}
We refer to such a space $\Omega\cross \T^n$,
with its standard symplectic form,
as an \define{Arnol'd--Liouville model.}  
\end{defn}

\begin{Theorem}[Arnol'd--Liouville]\labell{thm:A-L}
Every proper Lagrangian fibration with connected fibres
can be covered by Arnol'd--Liouville charts.
\end{Theorem}

\begin{proof}
This theorem is a reformulation of the Arnol'd--Liouville Theorem
in the case that the fibres are compact.
See, e.g., \cite[\S 49--50]{arnold:book}, 
\cite[lemma in \S 3]{markus-meyer}, \cite{duist-act-ang}.
\end{proof}

\begin{lemma}\labell{lemma:AL-isomorphisms}
Let
$$\Omega\cross \T^n \to \Omega' \cross \T^n \quad , \quad 
 (x,t) \mapsto (x',t') $$ 
be an isomorphism of Arnol'd--Liouville models, i.e., 
a fibre-preserving symplectomorphism. Then its components have the form
\[ x' = Ax+b, \qquad t'=A^{-T} t + g(x) \]
for some $A\in \GL_n(\Z)$, \ $b\in \R^n$, and
$g\colon \Omega \to \T^n$, 
where $A^{-T}$ denotes the inverse transpose of the matrix $A$.
\end{lemma}
(Note that our definition of an Arnol'd--Liouville chart assumes 
that $\Omega$ is connected.)
\begin{proof}[Proof of Lemma \ref{lemma:AL-isomorphisms}]
This follows from a direct computation 
(see for example Lemma 2.1 in~\cite{sepe}).
\end{proof}

\subsection*{Lagrangian torus fibration induces integral affine structure} \

\begin{lemma}\labell{lemma:volM=volB}
Let $\pi\colon (M,\omega) \to B$ be a $2n$-dimensional
Lagrangian torus fibration.
Then the set of maps $U\to \Omega$ that arise from Arnol'd--Liouville charts
is an integral affine atlas on $B$.
Moreover,
the induced measure on $B$ coincides with the push-forward
to $B$ of Liouville measure on $M$.
In particular, if $M$ is compact, then
$\vol(M) = \vol(B)$ (with respect to these measures). 
\end{lemma}

\begin{proof}
By Theorem~\ref{thm:A-L},
the domains of the maps $U \to \Omega$ are an open cover of $B$.
By Lemma~\ref{lemma:AL-isomorphisms},
the transition maps of this cover are integral affine.
By a direct computation in local coordinates in any Arnol'd--Liouville chart,
the push-forward of Liouville measure on $M$
is the measure on $B$.
\end{proof}

\subsection*{Enhanced Arnol'd--Liouville}\

In this subsection we will give a local model for a \emph{prequantized}
proper Lagrangian fibration, and prove an ``enhanced'' Arnol'd--Liouville
Theorem (Theorem~\ref{l:enhancedAL}).

\begin{defn}
Given a prequantized Lagrangian torus fibration,
i.e., a Lagrangian torus fibration $\pi\colon (M,\omega) \to B$,
equipped with a Hermitian line bundle $p \colon L \to M$
with connection $\nabla$ whose curvature is $\omega$, 
an \define{enhanced Arnol'd--Liouville chart} consists of:
\begin{itemize}
\item a connected open subset $U$ of $B$
\item a diffeomorphism of $U$ with an open subset $\Omega$ of $\R^n$
\item a lifting of this diffeomorphism to a symplectomorphism
from $\pi\inv(U)$ to $\Omega\cross \T^n$ with the
standard symplectic form $\sum dx_j \wedge dt_j$,
with $t_j$ mod $1$ coordinates on the $\T^n$ factor.
\item a lifting of this symplectomorphism to an isomorphism 
of Hermitian line bundles
from $p\inv(\pi\inv(U))$ to $\Omega \cross \T^n \cross \C$
that intertwines $\nabla$ with $d- 2 \pi i \sum x_j\, dt_j$.
\end{itemize}
We refer to such a complex line bundle 
$\Omega \cross \T^n \cross \C$ over $\Omega\cross \T^n$,
with the standard symplectic form on its base
and with the covariant derivative $d - 2\pi i \sum x_j\, dt_j$,
as an \define{enhanced Arnol'd--Liouville model}.
\end{defn}

\begin{Remark}\labell{rmk:conventions}
There are various different conventions in the literature for
curvature and prequantization; ours are as follows.
A 1-form $\alpha$ represents a connection with respect to a local
trivialization if, in this local trivialization, the covariant
derivative is $d + 2\pi i \alpha$.
The curvature of the connection is then $-d\alpha$.
A prequantization line bundle has a connection with curvature equal to
the symplectic form.
Thus, in an enhanced Arnol'd--Liouville chart,
the 1-form representing the connection is $-\sum x_j\, dt_j$
and the curvature is $\sum dx_j \wedge dt_j$.

Our convention for a 1-form representing a connection 
agrees with Kostant's paper \cite{Kost} (see Equation 1.4.3 on p.~98)
and differs from Woodhouse's book \cite{Wood}
by a sign and a factor of $2\pi$ (see \cite[\S A.3]{Wood}).
Given the 1-form, 
our convention for the curvature of a connection 
has opposite sign from \cite{Kost} (see p.~104, after Prop.~1.6.1)
and agrees with the conventions in~\cite{GGK} (see Lemma A.3, p.~169)
and \cite{Wood} (see \S A.3).
Given the curvature,
our convention for the prequantization data
agrees with~\cite{Kost} (see \S 4 on p.~165]) and with~\cite{sniatycki},
and 
differs from \cite{Wood} by a factor of $\hbar$ (see \cite[p.~158]{Wood}).

\eor
\end{Remark}

\begin{lemma}\labell{lemma:hol-enhanced-AL}
Let $\Omega\cross \T^n \cross \C$ be an enhanced Arnol'd--Liouville model,
with coordinates $(x_1,\ldots,x_n,\linebreak t_1,\ldots,t_n,z)$ 
with $t_j$ taken mod 1.  
For any fixed $x\in \Omega$ and any $k \in \{1,\ldots,n\}$, 
let $\gamma_k$ be the loop in $\{x \} \cross \T^n$ corresponding 
to one cycle in the $t_k$ variable.
Then the holonomy around $\gamma_k$ is $e^{2\pi i x_k}$.
\end{lemma}

\begin{proof}
We can take $\gamma_k$ to be given by 
$x_j(s)=x_j(0)$ for all $j$, \ $t_j(s) = 0$ for $j \neq k$, 
and $t_k(s) = s$, for $0 \leq s \leq 1$.
Because the covariant derivative is $d - 2\pi i\sum x_j\, dt_j$,
the horizontal lifts of $\gamma_k$ are then given by 
$z(s) = e^{2\pi i s x_k} z(0)$.
\end{proof}

\begin{thm}[Enhanced Arnol'd--Liouville] \labell{l:enhancedAL}
Every prequantized Lagrangian torus fibration
can be covered by enhanced Arnol'd--Liouville charts.
\end{thm}

{
\begin{proof}

By Theorem~\ref{thm:A-L}, every point in the symplectic manifold 
has a neighbourhood that can be identified
with an Arnol'd--Liouville model $\Omega \times \T^n$.
After shrinking the neighbourhood, we may assume
that $\Omega$ is contractible.
Thus, it is enough to prove that,
for any Arnol'd--Liouville model $\Omega \times \T^n$
with $\Omega$ a contractible open subset of $\R^n$,
any prequantization line bundle over it
is isomorphic to an enhanced Arnol'd--Liouville model.

Fix such an Arnol'd--Liouville model and a prequantization line bundle
$$ (L , \left< , \right> , \nabla) 
   \to (\Omega \times \T^n , \sum dx_j \wedge dt_j). $$

We first prove that $(L , \left< , \right> )$ is trivializable 
as a complex Hermitian line bundle.
Fix $q \in \Omega$. 
Because $\Omega$ is contractible, it is enough to show 
that the complex Hermitian line bundle $L|_{\{q\} \times \T^n}$ 
is trivializable.
Because the pullback of the curvature $\sum dx_j \wedge dt_j$
to the fibre $\{q\} \times \T^n$ is zero,
the pullback connection on $L|_{\{q\} \times \T^n}$ is flat. 
The holonomy of this connection becomes a map 
$$ \hol \colon \pi_1(\T^n) \to S^1 .$$
For each $j \in \{1,\ldots,n\}$, let $\gamma_j$ be a loop around $\T^n$
in the $t_j$ coordinate, and 
choose $h_j \in \R$ such that $\hol([\gamma_j]) = e^{2 \pi i h_j}$.
Then 
$\lambda := d - 2\pi i \sum h_j\, dt_j$ defines a connection on 
the trivial bundle $\T^n \times \C \to \T^n$ that has the same holonomy.
Because a flat bundle is determined up to isomorphism by its holonomy
(see Remark 1.12.1 in~\cite{Kost}),
$L|_{\{q\} \times \T^n}$ is isomorphic to this trivial bundle
(with the trivial Hermitian structure
and with the flat connection that we constructed).
In particular, $(L , \left< , \right>)$
is trivializable as a complex Hermitian line bundle.
Fix such a trivialization
$$\psi\colon L \to (\Omega\cross \T^n) \cross \C.$$

\bigskip
 
Since the curvature of $\nabla$ is $\sum dx_j \wedge dt_j$,
the trivialization $\psi$ 
takes $\nabla$ to $d - 2\pi i\sum x_j \, dt_j + 2\pi i\beta$ 
for some closed 1-form $\beta$ on $\Omega\cross \T^n$.
Since $\Omega$ is contractible, $H^1(\Omega\cross \T^n)$ is generated by
$dt_1,\ldots dt_n$, so $\beta$ has the form 
$\sum c_j\, dt_j + df$ for some constants $c_j$ and some function $f$ on
$\Omega \cross \T^n$.  Denote $c:=(c_1,\ldots,c_n)\in \R^n$. 
Then 
\[ (x,t,z) \mapsto \bigl(x, t, e^{- 2\pi i f(x,t)} z \bigr)\]
defines a bundle automorphism of $\Omega \times \T^n \times \C$ that takes 
$d - 2\pi i\sum x_j \, dt_j + 2\pi i\beta$ to $d - 2\pi i\sum (x_j-c_j) \, dt_j$.

We obtain an isomorphism of $L \to \Omega\cross \T^n$ with an
enhanced Arnol'd--Liouville model 
by composing $\psi$ with 
the bundle isomorphism
$$ (x,t,z) \mapsto (x-c,t,e^{-2\pi i f(x,t)}z).$$
\end{proof}
}

We next prove an ``enhanced'' analogue of Lemma~\ref{lemma:AL-isomorphisms},
giving the conditions for a map to preserve
enhanced Arnol'd--Liouville structures. 

\begin{lemma}\labell{lemma:enhanced-AL-isom}
Let
\begin{equation}\labell{eq:enhanced-coord-change}
 \Omega \times \T^n \times \C  \to  \Omega' \times \T^n \times \C
\quad , \quad (x,t,z) \mapsto (x',t',z')
\end{equation}
be an isomorphism of enhanced Arnol'd--Liouville models,
i.e., an isomorphism of line-bundles-with-connection
that lifts a fibre-preserving symplectomorphism
$ \Omega \times \T^n \to  \Omega' \times \T^n $.
Then its first two components have the form 
\begin{equation}\labell{eq:change-map}
  x' = Ax + b, \qquad t' = A^{-T} t + g(x) 
\end{equation}
for some $A\in \GL_n(\Z)$, $g\colon \Omega \to \T^n$, and $b\in \Z^n$.
\end{lemma}

\begin{proof}
By Lemma~\ref{lemma:AL-isomorphisms}, the first two components
have the form~\eqref{eq:change-map}
for some $A\in \GL_n(\Z)$, \ $g\colon \Omega \to \T^n$,
and $b\in \R^n$.  It remains to show that $b \in \Z^n$.

Writing the coordinates on 
$\Omega\cross \T^n \cross \C$ as  $(x_1,\ldots,x_n,t_1,\ldots,t_n,z)$,
with $t_j$ taken mod 1, 
for any fixed $x\in \Omega$ and any $j \in \{1,\ldots,n\}$, 
let $\gamma_j$ be the loop in $\{x \} \cross \T^n$ corresponding 
to one cycle in the $t_j$ variable.
By Lemma~\ref{lemma:hol-enhanced-AL},
the holonomy around $\gamma_j$ is $e^{2\pi i x_j}$.
Since the bundle $\{x\}\cross \T^n \cross \C$ is flat, the holonomy
is the same on homologous cycles, and thus gives a 
homomorphism $\hol_x \colon H_1(\{x\}\cross \T^n) \to S^1$.
Let $[\gamma_j]$ denote the homology
class of $\gamma_j$ in $H_1(\{x\}\cross \T^n)$.  

Similarly, for each $k=1,\ldots,n$ let $\gamma_k'$ be a loop in
a fibre $\{x'\}\cross \T^n$ of $\Omega'\cross \T^n \cross \C$
corresponding to one cycle in the $t_k'$ variable.

Under the map~\eqref{eq:enhanced-coord-change},
the loop $\gamma_j$ maps to a loop homologous to 
$\sum_k (A\inv)_{jk} [\gamma_k']$,
where as usual $(A\inv)_{jk}$ denotes the $(j,k)$ entry of the matrix $A\inv$.
Because the map~\eqref{eq:enhanced-coord-change} intertwines the connections, 
the holonomy of $\gamma_j$ equals the holonomy of its image
under the map~\eqref{eq:change-map}, which is given by 
\[ 
\hol_{x'}\Bigl(\sum_k (A\inv)_{jk} \gamma_k'\Bigr) =
  \exp \Bigl(2\pi i \sum_k (A\inv)_{jk} x_k'\Bigr). \]

Now $x_k' = \sum_i A_{ki} x_i + b_k$, so this becomes
\begin{align*}
  &= \exp \biggl(2\pi i \sum_k (A\inv)_{jk} \Bigl(\sum_i A_{ki} x_k + b_k\Bigr)\biggr)\\
  &= \exp \left(2\pi i \biggl( \sum_{k,i} (A\inv)_{jk} (A)_{ki} x_i + \sum_k (A\inv)_{jk} b_k \biggr)\right)\\
  &= \exp \biggl(2\pi i \bigl(x_j + (A\inv b)_j \bigr)\biggr).
\end{align*}
This will equal $e^{2\pi i x_j}$ iff $(A\inv b)_j$ is an integer.
Thus, the coordinate change map will preserve the holonomy of the connection
iff $A\inv b \in \Z^n$.
Since $A\inv \in \GL_n(\Z)$, this will be true iff $b\in \Z^n$.
\end{proof}

\subsection*{Prequantized Lagrangian torus fibration 
induces integral-integral affine structure} \ 

Finally, we show that a prequantization of a proper Lagrangian fibration
determines an integral-integral affine structure on the base.

\begin{lemma}\labell{lemma:prequant-implies-IIA}
Let $\pi\colon (M,\omega) \to B$ be a $2n$-dimensional
Lagrangian torus fibration with prequantization $(L,\nabla)$.  
Then the set of maps $U\to \Omega$ that arise from
enhanced Arnol'd--Liouville charts
is an integral-integral affine atlas on $B$.
Moreover, the set of integral points $B_\Z$ for this integral-integral
affine structure coincides with
the Bohr--Sommerfeld set of $(L,\nabla)$ in $B$.
\end{lemma}

\begin{proof}
By Theorem~\ref{l:enhancedAL}, $M$ can be covered by enhanced
Arnol'd--Liouville charts, 
and by Lemma~\ref{lemma:enhanced-AL-isom}, the coordinate change map
between the bases of two such charts
is locally an integral-integral affine transformation.
This proves the first claim.  

The second claim follows from Lemma~\ref{lemma:hol-enhanced-AL},
since a fibre is Bohr--Sommerfeld iff the holonomy is trivial around
a set of generators for the fundamental group.
\end{proof}

\begin{rmk}
The characterization of the Bohr--Sommerfeld set as points with integer
action coordinates has been around for some time.
More precisely, once the action coordinates have been fixed so that
one Bohr--Sommerfeld point has integer action coordinates,
the Bohr--Sommerfeld points in this chart are exactly those
points that have integer action coordinates.
(See Theorem 2.4 in~\cite{guil-st-gel-cet},
and the Remark at the end of \S3 of~\cite{hamilton-monod} for a brief
discussion of the subtleties.)
The existence of a globally defined Bohr--Sommerfeld set means that,
even if action coordinates cannot be globally defined,
it is possible to choose them in such a way that 
the condition ``all action coordinates integers'' is well-defined.
Such a collection of local action coordinates then gives an
integral-integral affine atlas.
\end{rmk}

\section{Proof of ``Upstairs Theorem'', assuming ``Downstairs Theorem''}
\labell{sec:RR=BS}

In this section, we prove our Upstairs Theorem, ``$\RR(M)=|\BS|$''
(which is Theorem~\ref{thm:RR=BS}), 
assuming our Downstairs Theorem, ``$\vol(B) = |B_\Z|$''
(which is Theorem~\ref{thm:volB=BZ}, which we prove
in section~\ref{sec:pf-volB=BZ}).
Originally, we had envisioned that our paper
would consist of essentially only this section,
before we realized that the Downstairs Theorem 
is not a triviality.

We begin by proving that, for a compact Lagrangian torus fibration, 
the Riemann--Roch number is equal to the volume.
Here, $[\omega]$ need not be integral.

\begin{lemma}\labell{lemma:RR=volM}
Let $\pi\colon M \to B$ be a Lagrangian torus fibration,
and assume that $M$ is compact.  Then $\RR(M) = \vol(M)$.
\end{lemma}

\begin{proof}
Recall $\RR(M) = \displaystyle \int_M e^\omega \Todd(M)$; thus,
it is enough to show that the Todd class of $M$ is trivial.
We proceed by a series of claims.

Let
$$ V := \ker \pi_* \subset TM $$ 
be the vertical bundle of the fibration.
Then we have a short exact sequence of real vector bundles over $M$:
\begin{equation*}
  0 \to V \to TM \to \pi^* T^* B \to 0 .
\end{equation*}
The map 
$$ u \mapsto u \contract \omega $$
defines an isomorphism 
$$ V \cong \pi^* T^*B $$
of real vector bundles over $M$.

Let $J$ be a compatible almost complex structure on $(M,\omega)$.
Because $V$ is Lagrangian, the map 
$$u+iv \mapsto u+Jv, \quad \text{ for each $y \in M$ and $u,v \in V_y$}, $$
defines an isomorphism 
$$ V\tensor \C \cong (TM,J).$$
of complex vector bundles over $M$.
Putting these together, we obtain an isomorphism 
\begin{equation}\labell{eq:TM-pistar-TB}
  (TM,J) \cong \pi^*(T^*B\tensor \C)
\end{equation}
of complex vector bundles over $M$.

The affine structure on $B$
induces a flat connection on the real vector bundle $TB$,
hence on its dual $T^*B$, 
hence on the complex vector bundle $T^*B \otimes \C$,
hence on its pullback, $\pi^*(T^*B \times \C)$, hence on $(TM,J)$.
Therefore the complex vector bundle $(TM,J)$ admits a flat connection,
and so its real Chern classes vanish, and so its Todd class is $1$.
Thus the Riemann--Roch number of $M$, as defined in~\eqref{eq:RR-defn}, 
satisfies
\begin{equation*}
  \RR(M,\omega) = \int_M \exp(\omega) = \int_M \frac{1}{n!} \omega^n = \vol(M).
\end{equation*}
\end{proof}

\begin{proof}[Proof of Upstairs Theorem 
assuming Downstairs Theorem] 
Let $M$ be a compact symplectic manifold,
$\pi\colon M \to B$ a Lagrangian fibration, and
$L\to M$ a prequantization line bundle.
By Lemma~\ref{lemma:prequant-implies-IIA}, these data induce 
an integral-integral affine structure on $B$
whose set of integral points coincides with the Bohr--Sommerfeld set
of $L$.  
By Lemma~\ref{lemma:RR=volM}, $\RR(M,\omega) = \vol(M,\omega)$.
By Lemma~\ref{lemma:volM=volB}, $\vol(M,\omega) = \vol(B)$.
So all we need to finish the proof of the Upstairs Theorem 
(Theorem~\ref{thm:RR=BS}) 
is to show that $\vol(B)$ is equal to the number of integral points in~$B_\Z$.
This is the content of the Downstairs Theorem (Theorem~\ref{thm:volB=BZ}), 
and so we are finished.
\end{proof}

\section{Lie torus bundles}
\labell{sec:lie-torus-bundles}

Recall that we use the term \define{Lie torus} 
for a Lie group that is isomorphic (as a Lie group) to $\T^n$ 
and the term \define{smooth torus}
for a manifold that is diffeomorphic to $\T^n$.
Accordingly, a \define{smooth torus bundle} is a fibre bundle
whose fibres are diffeomorphic to $\T^n$,
and a \define{Lie torus bundle} is a fibre bundle
whose fibres are equipped with Lie torus structures
and that admits local trivializations that respect these structures.
%

\subsection*{The integral lattice and the period lattice} \ 

An integral affine structure on $B$ gives rise to a lattice $\Lambda_q$
in each tangent space $T_qB$, varying smoothly with the basepoint $q \in B$:
take an integral affine chart $\phi\colon U \to \R^n$
whose domain $U$ contains~$q$, and define $\Lambda_q$ to be 
the pre-image of $\Z^n$ under the derivative 
$\phi_* \colon T_qB \to \R^n$ of the chart map.
Since the derivatives of coordinate changes are in $\GL_n(\Z)$, 
this condition is consistent on the overlaps of the domains of our charts,
and it defines a ``lattice bundle'' (namely, a bundle of lattices) 
$\Lambda$ in $TB$.
Suppressing the word ``bundle'', 
we refer to $\Lambda$ as the \define{integral lattice} 
of the integral affine manifold~$B$.
The \define{period lattice} of $B$ is the dual lattice, $\Lambda^*$, 
in $T^*B$:
\[ \xi \in \Lambda^*_q \quad \text{ iff } \quad \langle \xi, v \rangle \in \Z
\ \text{ for all } \ v \in \Lambda_q.\]

\begin{prop} \labell{prop:bundle-of-tori}
Let $B$ be an integral affine manifold,
and let $\Lambda \subset TB$ be its integral lattice.
Then $M^\vee := TB/\Lambda$, with its fibrewise Lie torus structure
induced from the vector bundle structure on $TB$,
is a Lie torus bundle over $B$.
Moreover, let $x_1,\ldots,x_n$ be integral affine local coordinates on $B$,
and let $x_1,\ldots,x_n,y_1,\ldots,y_n$ be the corresponding
adapted coordinates on $TB$.
Then $\Lambda = \{ y_j \in \Z \ \forall j \}$.
\end{prop}

\begin{proof}
The second part of the proposition follows from the definition of $\Lambda$
(and the definition of ``adapted coordinates'').
The first part of the proposition then follows
from the second part of the proposition:
the adapted coordinates, with the $y_j$s taken modulo $1$, 
give local trivializations of $M^\vee$ as a Lie torus bundle. 
\end{proof}

\subsection*{The affine lattice} \ 

An integral-\emph{integral} affine structure also gives rise to 
an ``affine lattice bundle'' $\Lambda^\aff$ in $TB$,
whose intersection $\Lambda^\aff_q$ 
with each tangent space $T_qB$ is a translation 
of the integral lattice $\Lambda_q$.
For each point $q \in B$ and integral-integral affine chart
$\phi\colon U \to \R^n$ with $x=\phi(q)$,
define $\Lambda^\aff_q$ to be the pre-image of $\Z^n$ under the map 
$x + \phi_* \colon T_qB \to \R^n$
(which is the linear approximation of $\phi$ at $q$). 
Equivalently,
$\Lambda^\aff_q = \{\sum ( - x_j + m_j) \dd{}{x_j} \st m_j \in \Z \}$.
This is independent of the choice of \iia\ chart.
Varying $q$, we refer to $\Lambda^\aff$ as the \emph{affine lattice}
of the integral-integral affine manifold $B$.


Because each $\Lambda^\aff_q$ is a $\Lambda_q$-coset in $T_qB$,
we can view $\Lambda^\aff$
as a section of the Lie torus bundle $TB/\Lambda$.
In adapted local coordinates, this section is parametrized
by 
\begin{equation} \labell{formula for s}
s(x) = (x,[-x]), 
\end{equation}
where $[x]$ denotes the equivalence class of $x$ mod $\Z^n$.
We also have the zero section $Z_0$, parametrized by $s_0(x) = (x,[0])$.
From this description in coordinates we see that these two sections meet
exactly above the points of $B_\Z$, 
that all their intersections are transverse, 
and that all their intersections have the same sign
(which depends only on the convention for the orientation 
of a tangent bundle).
Thus we can compute $\abs{B_\Z}$ by computing the intersection number of
$s(B)$ with $Z_0$. 
We will use this idea in the proof of
Theorem~\ref{thm:volB=BZ} in~\S\ref{sec:pf-volB=BZ}.

Informally, sections of the integral lattice are ``constant integers,''
while sections of the affine lattice ``have constant slope $-1$''
and intersect the zero section at integer points.
See Figure~\ref{fig:slope}.

\begin{figure}[h]
\centerline{\includegraphics[scale=0.5]{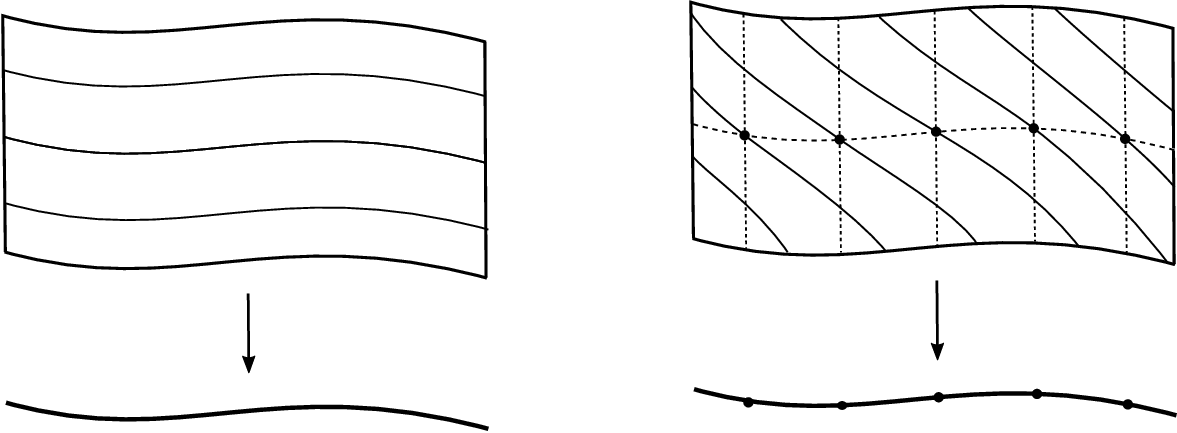}}
\caption{The integral lattice and the affine lattice}
\labell{fig:slope}
\end{figure}

\begin{Remark} \labell{rk:dual-torus}
In this section, given any integral-integral affine manifold $B$,
we constructed its affine lattice $\Lambda^\aff \subset TB$,
and we viewed it as a section of the Lie torus bundle $TB/\Lambda$.
In \S\ref{sec:defns 2}, we started with a prequantized
Lagrangian torus fibration $L \to M \to B$,
and we obtained from it an integral-integral structure on $B$.
When $B$ is obtained in this way,
we can identify the Lie torus bundle $TB/\Lambda$
as the dual torus bundle associated to $M \to B$,
and we can construct its section $\Lambda^\aff$
directly from the prequantization. 
For details, see Section \ref{sec:dual-torus-fibration}.
\eor
\end{Remark}

\section{Local torus actions}
\labell{sec:local torus actions}

A Lagrangian torus fibration $M \to B$ is not quite a principal torus bundle:
its fibres come with free and transitive torus actions,
but there is no one single torus that acts on all of these fibres.
Instead, 
the tori that act on these fibres assemble into a Lie torus bundle.
This Lie torus bundle comes with local trivializations,
which give torus actions on neighbourhoods of fibres of $M$,
and these fit together into a \emph{local torus action}.
In this section we provide details and 
we generalize some standard results on global torus actions
to local torus actions.

\subsection*{Local torus actions} \ 

Fix a Lie torus $T$, namely, a Lie group isomorphic to $(S^1)^n$.

\begin{defn}\labell{defn:local-torus-action}
A \define{$T$-atlas} on a manifold $M$ is a set of 
open subsets of $M$ that cover $M$ and are equipped with $T$-actions,
satisfying the following compatibility condition:
Given two such open subsets $U$ and $V$, 
the intersection $U\cap V$ is invariant under both actions,
and on each of its connected components,
the actions differ by an automorphism of $T$.
Two $T$-atlases are \define{equivalent} if their union is also a $T$-atlas,
namely, if their union satisfies the same compatibility condition. 
A \define{local $T$-action} on $M$ is an equivalence class of $T$-atlases.
\end{defn}

\begin{rmk}
There are several different definitions of ``local torus action'' in the
literature, none of which are exactly the same as ours; 
some others are in~\cite{yoshida-toric} and~\cite{cheeger-gromov},
\cite{gromov}.
\end{rmk}

\begin{rmk}
Our definition of local torus actions, and their resulting properties,
naturally extend to arbitrary compact Lie groups.
In this paper, we only need local torus actions that come from 
Lie torus bundles, as in the following example.
These can also be viewed as Lie groupoid actions.
\end{rmk}

\begin{Example}\labell{lemma:lie-torus-bundle-local-action}
Let $p\colon N \to B$ be a Lie torus bundle
(see \S\ref{sec:lie-torus-bundles})
with fibres isomorphic to $T$.
A local trivialization $p\inv(U) \cong U \cross T$ 
induces a natural action of $T$
on $p\inv(U)$ by acting on the second component of $U\cross T$. 
Because the automorphism group of the Lie torus $T$ is discrete,
for any two such local trivializations,
the $T$ actions on the (connected components of the)
intersection of their domains
differ by an automorphism of $T$.
This equips $N$ with a local $T$-action.
\end{Example}

We say that a $k$-form $\alpha \in \Omega^k(M)$ is \define{invariant}
with respect to the local torus action
if the restriction of $\alpha$ to every element of every $T$-atlas
is $T$-invariant.
(Because being $T$-invariant is preserved under automorphisms of $T$,
it is enough to require the existence of one $T$-atlas
such that restrictions of $\alpha$ to its elements are $T$-invariant.)
Denote by $\Omega^k(M)^{T}$ the space of invariant $k$-forms on~$M$. 

Next, define an ``averaging'' map
$\av\colon \alpha \mapsto \ol{\alpha} \colon
\Omega^k(M) \to \Omega^k(M)^{T}$
as follows.
In any open set $U$ in a T-atlas,
it is defined by integrating over the torus: 
\[  \ol\alpha|_U = \int_{t\in T} t^* \alpha|_U \, d\mu \]
where $d\mu$ is normalized Haar measure on $T$.
This construction agrees on overlaps between different elements of a T-atlas,
since the Haar measure is invariant under $\Aut(T)$,
and so this operation is well-defined over all of $M$.  
Note that $\alpha$ is invariant iff $\bar{\alpha} = \alpha$.

Invariance and averaging are compatible with the exterior derivative:

\begin{lemma} \labell{d of bar}
Let $M$ be a manifold equipped with a local $T$-action.
For any $\alpha \in \Omega^k(M)$, we have $\overline{d\alpha} = d\bar{\alpha}$. 
  Thus, if $\alpha$ is invariant, then $d\alpha$ is invariant,
and if $\alpha$ is closed, then $\bar{\alpha}$ is closed.
\end{lemma}

\begin{proof}
The argument is similar to the known argument for a global $T$ action;
we include it here for completeness.

In a local coordinate chart contained in an element of a $T$-atlas, 
write $t^*\alpha = \sum_I a_I(x,t) dx_I$.
Then 
\[\ol{\alpha} = \sum_I 
\Big( \int_T a_I(x,t) dt \Big) dx_I ,\]
where the integral is with respect to the normalized Haar measure on $T$.
Similarly, 
\[\ol{d\alpha} = \sum_{I,j} 
\Big( \int_T \dd{a_I}{x_j}(x,t) dt \Big)
dx_j \wedge dx_I .\]
So it suffices to show that, for each $I$ and $j$,
\[ 
\int_T \dd{a_I}{x_j}(x,t) \, dt 
 = \dd{}{x_j} \left(\int_T a_I(x,t) \, dt\right) .\]
This, in turn,
follows from Theorem 9.42 in~\cite{rudin}, differentiation under the
integral sign, which requires only continuity of $\dd{a_I}{x_j}$.
Since each point has a neighbourhood where $\ol{d\alpha} = d\ol{\alpha}$,
we are done.

For the second claim, recall 
that $\beta$ is invariant if and only if $\ol{\beta} = \beta$,
and  note that $\ol{0} = 0$. 
\end{proof}

The key fact we will need about local torus actions
is that, when the manifold is compact,
any cohomology class has an invariant representative.
We prove this in Theorem~\ref{thm:invariant-representative}.
We expect this fact to remain true for non-compact manifolds; 
it should follow from the fact that the subspace of exact forms
is closed (say, in the $C^\infty$ topology)
in the space of all differential forms. 
For a \emph{global} action on a non-compact manifold,
an easier proof is given in~\cite[\S9]{onishchik}.

\begin{thm}\labell{thm:invariant-representative}
Let $M$ be a compact manifold equipped with a local $T$ action.
Then, for each closed $k$-form $\alpha$, 
its average $\ol{\alpha}$ is closed,
and the cohomology classes of $\ol{\alpha}$ and $\alpha$ are equal in $H^k(M)$.
In particular, any cohomology class in $H^k(M)$
has a representative in $\Omega^k(M)^{T}$.
\end{thm}

\begin{proof}
Let $M$ be a compact manifold equipped with a local $T$ action.
Denote $H^k(\Omega^\bullet(M)^{T},d)$ by $H^k_{\Inv}(M)$.
Since the condition of invariance is preserved by the exterior derivative,
i.e., $d\bigl(\Omega^k(M)^{T}\bigr) \subseteq \Omega^{k+1}(M)^{T}$,
the inclusion map $\incl \colon \Omega^k(M)^{T} \into \Omega^k(M)$
induces a map on cohomology,
\[ \incl_*\colon H^k_{\Inv}(M) \to H_{\dR}^k(M). \]
Since the averaging map $\av \colon \alpha \mapsto \ol{\alpha}$ 
also commutes with $d$,
it also induces a map on cohomology
\[ \av_* \colon H_{\dR}^k(M) \to H^k_{\Inv}(M).\]
Note that the composition $\av_* \iota_* $
is the identity map on $H^k_{\Inv}(M).$

Fix a finite $T$ atlas $\{ U_1,\ldots,U_m\}$ on $M$.
We will prove by induction on $\ell$ that the map
\begin{equation}\labell{induction assumption}
\iota_*\colon H_{\Inv}^k\left( \cup_{i=1}^\ell U_i\right)\to
H^k_\dR\left( \cup_{i=1}^\ell U_i ;\R\right)
\end{equation}
is an isomorphism for all $k$. 

The case $\ell=1$ follows from the fact that,
because $T$ is a compact connected Lie group
and $U_1$ is a $T$ manifold,
$H^k_{\dR}(U_1;\R)=H^k\left( \Omega^\bullet (U_1)^T ,d\right)$; 
see~\cite[\S9]{onishchik}.

Next, fixing $1 \leq \ell < m$,
suppose that \eqref{induction assumption} is an isomorphism for all $k$,
and we would like to prove that
$$ \iota_*\colon H_{\Inv}^k\left( \cup_{i=1}^{\ell+1} U_i\right)\to
H^k_\dR\left( \cup_{i=1}^{\ell+1} U_i ;\R\right)$$
is an isomorphism for all $k$.

Consider the following commutative
diagram of Mayer--Vietoris long exact sequences
\[
\xymatrix{ \ar[r]
  & H_{\Inv}^k\Bigl( \bigl(\cup_{i=1}^\ell U_i\bigr)\cup U_{\ell+1}\Bigr)
  \ar[r]\ar[d]^{\iota_*}
  & H_{\Inv}^k\left(\cup_{i=1}^\ell U_i\right)\oplus H_{\Inv}^k(U_{\ell +1})
  \ar[r]\ar[d]^{\iota_*\oplus \iota_*} &
  H_{\Inv}^k\Bigl( \bigl(\cup_{i=1}^\ell U_i\bigr)\cap U_{\ell+1}\Bigr)
  \ar[r]\ar[d]^{\iota_*} & \\
  \ar[r]
  & H^k_{\dR}\Bigl( \bigl(\cup_{i=1}^\ell U_i\bigr)\cup U_{\ell +1} \Bigr)\ar[r] 
  & H^k_{\dR}\left(  \cup_{i=1}^\ell U_i\right)\oplus H^k_{\dR}(U_{\ell +1})\ar[r]
  & H^k_{\dR}\Bigl( \bigl(\cup_{i=1}^\ell U_i\bigr)\cap U_{\ell +1} \Bigr)\ar[r] & 
}
\]
Since each $U_j$ is $T$-invariant, so is the intersection
$\bigl(\cup_{i=1}^\ell U_i\bigr)\cap U_{\ell +1}$.
Because we have an actual $T$ action on this intersection,
\begin{equation}\labell{isomorphism intersections}
\iota_*\colon H_{\Inv}^k\Bigl( \bigl(\cup_{i=1}^\ell U_i\bigr)\cap U_{\ell +1} \Bigr)
\to  H^k_{\dR}\Bigl( \bigl(\cup_{i=1}^\ell U_i\bigr)\cap U_{\ell +1} \Bigr)
\end{equation}
is an isomorphism, giving the right-hand vertical arrow in the above diagram. 
The middle vertical arrow is the map $\iota_*\oplus \iota_*$;
the first $\iota_*$ is an isomorphism by the induction hypothesis,
and the second is an isomorphism because there is a global $T$-action on
$U_{\ell+1}$.
Since these maps are isomorphisms for all $k$, the Five Lemma implies that
the left-hand vertical map 
\[
\iota_*\colon H_{\Inv}^k\Bigl( 
       \bigl(\cup_{i=1}^\ell U_i\bigr) \cup U_{\ell+1}\Bigr)
\to
  H_{\dR}^k\Bigl( 
       \bigl(\cup_{i=1}^\ell U_i\bigr) \cup U_{\ell+1}\Bigr)
\]
is also an isomorphism.

This completes the induction, and proves
$ \incl_*\colon H^k_{\Inv}(M) \to H_{\dR}^k(M)$
is an isomorphism.

Finally, since $\av_*$ is a left inverse of $\iota_*$
and $\iota_*$ is an isomorphism, $\av_*$ is the inverse of $\iota_*$. 
\end{proof}

\section{Proof of ``Downstairs Theorem''} 
\labell{sec:pf-volB=BZ}

The purpose of this section is to prove Theorem~\ref{thm:volB=BZ},
whose statement we now recall:

\begin{thm-nn}[\ref{thm:volB=BZ}, restated]
Let $B$ be a compact integral-integral affine manifold.
Then the volume of $B$ is equal to the number of integral points in $B$:
\[ \vol(B) = \abs{B_\Z}.\]
\end{thm-nn}

Let $B$ be an integral-integral affine manifold.  
After passing to a double covering, if necessary, we fix an orientation of $B$.
(Note that the double covering inherits 
an integral-integral affine structure,
and that if Theorem \ref{thm:volB=BZ} holds for the double covering,
then it holds for the original $B$.)

Let $\Lambda \subset TB$ be the integral lattice.
Then $M^\vee := TB/\Lambda$ is a Lie torus bundle over $B$
(see Proposition~\ref{prop:bundle-of-tori});
let $\pi^\vee\colon M^\vee \to B$ denote the projection.
Let $Z_0 \subset M^\vee$ be induced from the zero section of $TB$.

Given adapted coordinates $x_1,\ldots,x_n, y_1,\ldots,y_n$ on $TB$
as in Proposition~\ref{prop:bundle-of-tori},
taking $y_j$ mod 1 gives what we call ``adapted coordinates'' on $M^\vee$.

\begin{lemma} \labell{exists dyn}
There exists a unique closed $n$-form $dy^n$ on $M^\vee$ such that,
in any adapted coordinates on $M^\vee$ as above, we have
$dy^n = dy_1 \wedge \ldots \wedge dy_n$.
\end{lemma}

\begin{proof}
  Given two different sets $(x_1,\ldots,x_n, y_1,\ldots,y_n)$
  and $(x_1',\ldots,x_n', y_1',\ldots,y_n')$ 
  of coordinates on $M^\vee$ as above, 
the $x$ and $x'$ coordinates will locally differ by 
$x' = Ax + b$ with $A \in \GL_n(\Z)$ and $b \in \R^n$, 
and the $y$ coordinates will locally differ by $y' = Ay$
with the same $A \in \GL_n(\Z)$.
Because the $x$ and $x'$ coordinates are oriented,
the determinant of $A$ is positive, hence it is $1$.
It follows that
$dy_1 \wedge \ldots \wedge dy_n = dy_1' \wedge \ldots \wedge dy_n'$.
\end{proof}

We orient $M^\vee$ such that 
every set of coordinates $x_1,\ldots,x_n,y_1,\ldots,y_n$ as above is oriented.

\begin{Proposition} \labell{p:poincare}
The cohomology class $[dy^n]$ is the Poincar\'e dual 
of the submanifold $Z_0$. 
\end{Proposition}

\begin{proof}
We need to show that, given any closed $n$-form $\alpha$ on $M^\vee$,
\begin{equation}\labell{eq:condition-for-PD}
  \int_{Z_0} \alpha = \int_{M^\vee} \alpha \wedge dy^n .
\end{equation}

As in Example~\ref{lemma:lie-torus-bundle-local-action},
the Lie torus bundle $M^\vee$ has a local $T$-action,
which in coordinates as above is given by translations
on the $y_j$ (mod 1) coordinates.
By Theorem~\ref{thm:invariant-representative}, 
we may assume that $\alpha$ is invariant under this local torus action.

Let $\{ U_k\}$ be an open cover of $B$
such that on each $U_k$ there exists an integral affine chart,
and let $\{\rho_k\}$ be a partition of unity subordinate to this cover.
Let $\alpha_k := (\pi^\vee)^* \rho_k \cdot \alpha$; then 
\[
\int_{Z_0} \alpha = \sum_k \int_{Z_0} \alpha_k
\qquad \text{ and } \qquad
\int_{M^\vee} \alpha \wedge dy^n
= \sum_k \int_{M^\vee} \alpha_k \wedge dy^n. \]
The $n$-forms $\alpha_k$ are still invariant 
(they might no longer be closed, but this will end up not creating
a problem), 
so it remains to show that \eqref{eq:condition-for-PD}
holds when $\alpha$ is replaced by the invariant $n$-form $\alpha_k$.

In adapted coordinates over $U_k$, the invariant $n$-form $\alpha_k$
is a sum of terms of the form
$c(x)\, dx_{i_1} \wedge \ldots \wedge dx_{i_r}
\wedge dy_{j_1} \wedge \ldots \wedge dy_{j_s}$ with $r+s = n$.
Taking its wedge with $dy^n$, all terms with $s>0$  vanish,
and we obtain
\[ \alpha_k \wedge dy^n = c_k(x)\, dx_1 \wedge \ldots \wedge dx_n
\wedge dy_1 \wedge \ldots \wedge dy_n \]
for some smooth $c_k$ depending only on $x$.

In these same coordinates, the submanifold $Z_0$ is parametrized by the map 
$$ s_0 \colon U_k \to Z_0 \quad , \quad x \mapsto (x,[0]) ,$$ 
and the pullback of $\alpha_k$ by this map is given by 
$$ s_0^* \alpha_k = c_k(x)\, dx_1 \wedge \cdots \wedge dx_n $$ 
with the same coefficient $c_k(x)$.  So
\begin{align*}
\int_{M^\vee} \alpha_k \wedge dy^n
&= \int_{U_k \times (\R/\Z)^n}  c_k(x)\,
  dx_1 \wedge \ldots \wedge dx_n \wedge dy_1 \wedge \cdots \wedge dy_n\\
&= \left(\int_{U_k} c_k(x) \, dx_1\wedge \cdots \wedge dx_n \right)
  \left(\int_{ (\R/\Z)^n } dy_1\wedge\cdots\wedge dy_n \right) \\
&= \int_{U_k} c_k(x) \, dx_1\wedge \cdots \wedge dx_n  \\
&= \int_{Z_0} \alpha_k, 
\end{align*}
as required.
\end{proof}

Finally, recall from Section~\ref{sec:lie-torus-bundles}
that $M^\vee$ comes with a section $D \subset M^\vee$ 
that in local coordinates as above
is given by $D = \{ (x,[-x]) \} $.
The orientation of $B$ induces orientations of $Z_0$ and of $D$, 
and as noted in Section~\ref{sec:lie-torus-bundles},
$D$ intersects $Z_0$ above $B_\Z$, so
$\lvert D \cap Z_0\rvert = \lvert B_\Z \rvert$.
The intersection number of $Z_0$ with $D$ 
is then $(-1)^n \lvert D \cap Z_0\rvert$.

The pullback of $dy^n$ by the map $s(x) = (x,[-x])$
that parametrizes $D$ 
is $(-1)^n dx_1\wedge \ldots \wedge dx_n$.
We then have
\begin{multline*}
\abs{B_\Z} = (-1)^n \big( [D] \cup [Z_0] \big) = (-1)^n \int_{D} \PD(Z_0) 
 = (-1)^n \int_D dy^n
 \\ 
 = \int_B (-1)^n s^* dy^n = \int_B dx_1 \wedge \cdots \wedge dx_n = \vol(B).
\end{multline*}
This completes the proof of Theorem~\ref{thm:volB=BZ}.

\bigskip

\section{Dual torus fibrations} 
\labell{sec:dual-torus-fibration}

In this section we collect some material that 
we don't strictly need for our proof, and that is well-known, 
but that may help clarify the geometric context of the structures
that we have been considering.

In the proof of our Downstairs Theorem in \S\ref{sec:pf-volB=BZ},
a key ingredient was a section of $TB/\Lambda$, constructed from the
integral-integral affine structure, whose intersections with the zero section
are the integer points of $B$.
From the ``upstairs'' viewpoint,
this section is determined by the prequantization line bundle,
and its intersections with the zero-section are the Bohr--Sommerfeld points.  
In this section we briefly explain how to define the section directly
from the line bundle, and how it corresponds with the construction via
integral-integral affine structures.  

A \define{flat circle bundle} is a principal circle bundle
equipped with a flat connection.
For any manifold $N$,
there is a natural bijection between 
the set of equivalence classes of flat circle bundles over $N$
and the set of isomorphism classes of
complex Hermitian line bundles over $N$ equipped with flat Hermitian 
connections.

For any smooth torus\footnote{This works for any compact connected manifold,
but we don't need this generality, and the term ``dual torus'' doesn't
sound right when we don't start with a torus.} 
$N$, we denote by $N^\vee$
the set of equivalence classes of flat circle bundles over~$N$.
This set is naturally a Lie torus, e.g., by using the holonomy 
to identify $N^\vee$ with $\Hom(\pi_1(N),S^1)$ (for any choice of basepoint).
We will call $N^\vee$ the \define{dual torus} to $N$.
Any diffeomorphism $N_1 \cong N_2$ 
naturally induces an isomorphism of Lie tori $N_1^\vee \cong N_2^\vee$.

For any torus $T$ with Lie algebra $\Lt$,
the exponential map $\Lt \to T$ is a group homomorphism,
the \emph{integral lattice} $\Lt_\Z$ is its kernel, 
and we obtain identifications $T = \Lt/\Lt_Z$ and $\pi_1(T)=\Lt_\Z$.
The dual lattice $\Hom(\Lt_\Z,\Z)$ embeds in $\Lt^* = \Hom(\Lt,\R)$ 
through tensoring with $\R$, yielding the \emph{weight lattice} 
$\Lt_\Z^* \subset \Lt^*$.
Writing elements of $S^1$ as $e^{2\pi i t}$, 
we obtain an identification $S^1=\R/\Z$, which yields
$\Lt_\Z = \Hom(S^1,T)$ and $\Lt_\Z^* = \Hom(T,S^1)$.

With these identifications,
$T^\vee = \Hom(\pi_1(T),S^1) = \Lt^*/\Lt_\Z^*$.\footnote{
Often in the literature, the phrase \emph{dual torus} 
is used only in the case that $N=T = \Lt/\Lt_\Z$
and is defined as $T^\vee=\Lt^*/\Lt_\Z^*$.} 
For any $T$-torsor\footnote
{A $T$-torsor is a smooth manifold equipped with a free and transitive 
$T$-action. Such a manifold is in particular a smooth torus.}
$N$, the natural identification of $\pi_1(N)$ with $\pi_1(T)$
yields $N^\vee = T^\vee = \Lt^*/\Lt_\Z^*$.

For the standard torus $\T^n=(S^1)^n$, 
and with the standard identification $\R^n = (\R^n)^*$,
we obtain the identifications $(\T^n)^\vee = \R^n/\Z^n = \T^n$.
In this identification, the element $(a_1,\ldots,a_n)$ of $\T^n$
on the right hand side
corresponds to the element of $(\T^n)^\vee$ on the left hand side
that is represented by a flat connection 
whose holonomony along each loop $\gamma_j$ is $a_j$,
where $\gamma_j$ is the loop in the $t_j$ coordinate.

Let $ \pi \colon M \to B$ be a (smooth) torus fibration.
As a set, the \define{dual torus fibration $M^\vee \to B$}
consists of the disjoint union of the dual tori;
namely, for each $x \in B$, the fibre $(M^\vee)_x$
is the set of equivalence classes of flat circle bundles over the fibre $M_x$.
The local trivializations $M|_U \cong U \times \T^n$
induce a manifold structure on $M^\vee$ and local trivializations 
$M^\vee|_U \cong U \times \T^n$,
making $M^\vee$ into a Lie torus fibration.

Given a Lagrangian torus fibration $M \to B$
and a prequantization line bundle $L \to M$,
the restriction of $L$ to each Lagrangian torus fibre $M_x$ is flat,
thus it determines an element of the corresponding dual torus,
which is trivial if and only if $M_x$ is a Bohr--Sommerfeld fibre
(by definition).
Thus, $L$ induces a section 
$$s_L \colon B\to M^\vee$$
of the dual torus bundle $M^\vee$,
which intersects the zero section exactly above the Bohr--Sommerfeld points.
Up to sign, this section coincides with the 
affine lattice that we introduced
in Section \ref{sec:lie-torus-bundles}, in a sense that we will now spell out.

\begin{Proposition} \labell{prop:identify lattices}
\begin{enumerate}
\item \labell{part1} 
Given a Lagrangian torus fibration $M \to B$,
the dual fibration $M^\vee$ is naturally isomorphic as a Lie torus bundle
to $TB/\Lambda$, where $\Lambda$ is the integral lattice arising from
the induced integral affine structure on $B$ 
(as in \S\ref{sec:lie-torus-bundles}).

\item \labell{part2} 
Let $L\to M\to B$ be a prequantized Lagrangian torus fibration,
and let $\Lambda^{\text{aff}}$ be the affine lattice arising from the induced 
integral-integral affine structure on $B$
(as in \S\ref{sec:lie-torus-bundles}).
Then the isomorphism $M^\vee \cong TB/\Lambda$ of part \eqref{part1}
identifies $s_L$ with $-\Lambda^{\text{aff}}$ viewed as a section
of $TB/\Lambda$.
\end{enumerate}
\end{Proposition}

\begin{proof}[Sketch of proof]
Part \eqref{part1} is obtained from the proof of the Arnol'd--Liouville theorem.
Namely, by taking time-1 Hamiltonian flows of collective functions
(pullbacks to $M$ of smooth functions on $B$), we obtain for each $x \in B$
an action of $T_x^*B$ on the fibre $M_x$. 
This action has stabilizer $\Lambda^*_x$, so makes $M_x$
into a $T_x^*B/\Lambda^*_x$-torsor.
Therefore we can identify $M_x^\vee$ with the torus 
dual to $T_x^*B/\Lambda^*_x$, namely, with $T_xB/\Lambda_x$.

Part \eqref{part2} then follows from the coordinate expression of $s_L$ 
in an enhanced Arnol'd--Liouville chart.  Here are some details:
By ``enhanced Arnol'd--Liouville'' (Theorem~\ref{l:enhancedAL}),
we may assume that $L = \Omega \times \T^n \times \C$,
with coordinates $(x_1,\ldots,x_n,t_1,\ldots,t_n,z)$ with $t_j$ taken mod $1$.
In particular, we identify $M = \Omega \times \T^n$ and $B = \Omega$.
From these identifications, 
we obtain $M^\vee = \Omega \times (\T^n)^\vee$
and 
$TB = \Omega \times \R^n$ and $\Lambda = \Omega \times \Z^n$,
so $TB/\Lambda = \Omega \times \T^n$.
Composing with our identification $(\T^n)^\vee = \T^n$,
we obtain the identification $TB/\Lambda=M^\vee$ of Part~\ref{part1}.
Finally, by Lemma~\ref{lemma:hol-enhanced-AL},
the holonomy along $\gamma_j$
of the connection in the enhanced Arnol'd--Liouville chart
is $e^{2\pi i x_j}$;
using our identification $(\T^n)^\vee = \T^n$,
we obtain $s_L(x_1,\ldots,x_n) = 
  (x_1,\ldots,x_n,e^{2\pi i x_1},\ldots,e^{2\pi i x_n})$,
also written as $(x,[x])$, which is the negative 
of the expression
in coordinates for the section $\Lambda^\aff$ \eqref{formula for s}.
\end{proof}

\section{Comments on the literature}
\labell{sec:literature}

The results and ideas that we present in this paper are in principle
known to experts.
In particular, our Upstairs Theorem
follows from a more general theorem of Fujita, Furuta, and Yoshida,
and appears as Corollary~6.12 in their paper~\cite{fujita-furuta-yoshida}.
They, in turn, attribute this result to Andersen's paper~\cite{andersen};
however, Andersen's paper does not explicitly contain this result.
Corollary~4.1 of Andersen's paper \cite{andersen} is similar,
and its proof contains ingredients that are similar to ours,
but this proof lacks detail.

Specifically, 
Andersen's Corollary~4.1 should follow from the proof of his Theorem~4.1.
In the proof of his Theorem~4.1, 
Andersen describes an $n$-form $\alpha$ on a torus bundle
(which is the same as the $n$-form $dy^n$ on $M^\vee$ in our 
Lemma~\ref{exists dyn}),
and writes: ``It is clear that $\alpha$ is the Poincar\'e dual
of the zero section''.  To us, this was not ``clear;'' 
this result is our Proposition~\ref{p:poincare},
whose proof relies 
on our nontrivial Theorem~\ref{thm:invariant-representative}.

Andersen \cite{andersen} also works with the half-form correction 
and we don't (yet), but this difference should be minor.
Finally, Andersen works with so-called mixed polarizations, of which 
our Lagrangian fibrations are a very special case;
it would be interesting to extend our work
to Andersen's more general setup.

Our proof of Theorem~\ref{thm:volB=BZ} is inspired by 
Gross's ``Mirror Riemann--Roch conjecture'' \cite[Conjecture 4.9]{gross}.
In the presence of a (special) Lagrangian torus fibration
and a (``permissible'') dual fibration,
this conjecture expresses the Euler characteristic
of a holomorphic line bundle over one of the fibrations
as an intersection number of sections of the other fibration, 
up to sign.

Similar ideas also appear in Tyurin's paper \cite{tyurin}.
Tyurin works with a singular Lagrangian fibration 
of a compact K\"aler manifold $M$,
with a prequantization line bundle $M \to L$.
He refers to the question on whether 
the dimension of the space of holomorphic sections 
equals the number of Bohr--Sommerfeld fibres for $L^{\otimes k}$
as the ``numerical quantization problem'' \cite[p.~6]{tyurin}.
(About this dimension, 
see our discussion in Section~\ref{sec:defns} under ``Riemann--Roch numbers''.)
He sketches a general argument for attacking this problem
by realizing the Bohr--Sommerfeld set as the intersection of two sections 
of a dual fibration.
(The proof of the ``Downstairs Theorem'' in our Section~\ref{sec:pf-volB=BZ}
uses a similar idea.)
Tyurin applies this general method to the cases of
elliptic curves (Prop 1.4 and Cor 1.1), K3 surfaces (Theorem 2.2),
and ``conditionally'' to Calabi--Yau 3-folds (Prop 3.2 and Cor 3.1; 
see the first paragraph of~\S3 for a discussion of the ``conditionality''). 
He acknowledges that ``these constructions have been developed and
used by a number of physicists and mathematicians.''

Tyurin works with K\"ahler manifolds, and allows the polarizations to have
nodal (``focus-focus'') singularities (see (2.54) on p.\ 30).
In our setting, we do not allow singularities in our polarizations;
on the other hand, integrability of almost complex structures is not assumed.
In particular, our proof applies to the Kodaira--Thurston manifold,
but the argument in~\cite{tyurin} doesn't.
It would be interesting to extend our work to allow focus-focus singularities.

\bigskip \bigskip \bigskip


{\ifdebug {\newpage} \else {\end{document}} \fi}

\end{document}